\documentclass[12pt,letterpaper,reqno]{amsart}
\usepackage{amsmath,amsthm,amsfonts,amssymb,mathtools}
\usepackage[lite]{amsrefs}

\newtheorem{theorem}{Theorem}
\newtheorem{proposition}[theorem]{Proposition}
\newtheorem{lemma}[theorem]{Lemma}

\numberwithin{equation}{section}

\newcommand{\M}{{\mathcal M}}

\begin{document}

\title[Sobolev norm estimates for a class of bilinear multipliers]{Sobolev norm estimates for a class of\\ bilinear multipliers}

\author{Fr\'{e}d\'{e}ric Bernicot}
\address{Fr\'{e}d\'{e}ric Bernicot, CNRS -- Universit\'{e} de Nantes, Laboratoire de Math\'{e}matiques Jean Leray, 2,
Rue de la Houssini\`{e}re F-44322 Nantes Cedex 03, France}
\email{frederic.bernicot@univ-nantes.fr}
\thanks{The first author is partially supported by the ANR under the project AFoMEN no. 2011-JS01-001-01.}

\author{Vjekoslav Kova\v{c}}
\address{Vjekoslav Kova\v{c}, University of Zagreb, Department of Mathematics, Bijeni\v{c}ka cesta 30, 10000 Zagreb, Croatia}
\email{vjekovac@math.hr}
\thanks{The second author was partially supported by the MZOS grant 037-0372790-2799 of the Republic of Croatia.}


\day=30 \month=07 \year=2013
\date{\today}

\begin{abstract}
We consider bilinear multipliers that appeared as a distinguished particular case in the classification of
two-dimensional bilinear Hilbert transforms by Demeter and Thiele \cite{DT}.
In this note we investigate their boundedness on Sobolev spaces.
Furthermore, we study structurally similar operators with symbols that also depend on the spatial variables.
The new results build on the existing $\mathrm{L}^p$ estimates for a paraproduct-like operator
previously studied by the authors in \cite{Ber} and \cite{Kov}.
Our primary intention is to emphasize the analogies with Coifman-Meyer multipliers
and with bilinear pseudodifferential operators of order $0$.
\end{abstract}

\maketitle

\section{Introduction}

Multilinear Fourier multipliers that are singular along linear subspaces of relatively high codimension have attracted some attention since the
breakthrough papers of Lacey and Thiele \cite{LT1},\,\cite{LT2} on $\mathrm{L}^p$ boundedness of the \emph{bilinear Hilbert transform},
\begin{equation}\label{eqbht}
T_{\mathrm{BH}}(f,g)(x)
:= \mathrm{p.v.}\!\int_{\mathbb{R}}f(x-t)g(x+t)\frac{dt}{t}
= \int_{\mathbb{R}^2} \pi i \mathop{\mathrm{sgn}}(\eta-\xi)
e^{2\pi i x (\xi+\eta)} \widehat{f}(\xi) \widehat{g}(\eta) d\xi d\eta .
\end{equation}
A more general result for multilinear singular multipliers obtained using similar techniques can be found for instance in a paper
by Muscalu, Tao, and Thiele \cite{MTT}.
Recently, Demeter and Thiele discovered certain new phenomena in higher dimensions in an attempt to classify
and study two-dimensional analogues of (\ref{eqbht}).
Their paper \cite{DT} is concerned with $\mathrm{L}^p$ estimates for two-dimensional bilinear multipliers
\begin{equation}\label{eqbilinfourmulti}
T_{\mu}(f,g)(x,y) := \int_{\mathbb{R}^4} \mu(\xi_1,\xi_2,\eta_1,\eta_2)
e^{2\pi i (x (\xi_1+\eta_1) + y (\xi_2+\eta_2))} \widehat{f}(\xi_1,\xi_2) \widehat{g}(\eta_1,\eta_2) d\xi_1 d\xi_2 d\eta_1 d\eta_2 ,
\end{equation}
with a symbol
$$ \mu(\xi_1,\xi_2,\eta_1,\eta_2) = m\big(A^\tau(\xi_1,\xi_2)+B^\tau(\eta_1,\eta_2)\big) , $$
where $A,B\in\mathrm{M}_{2}(\mathbb{R})$ are linear maps on $\mathbb{R}^2$
and $m\in\mathrm{C}^{\infty}\big(\mathbb{R}^2\!\setminus\!\{(0,0)\}\big)$ satisfies
\begin{equation}\label{eqhomsymbolest}
\big|\partial_{\tau_1}^{\beta_1}\partial_{\tau_2}^{\beta_2}m(\tau_1,\tau_2)\big| \leq C_{\beta_1,\beta_2} (|\tau_1|+|\tau_2|)^{-\beta_1-\beta_2}
\quad\textrm{for all }\beta_1,\beta_2\geq 0.
\end{equation}
Note that the singularity of $\mu$ can be a subspace of $\mathbb{R}^4$ of codimension $2$,
making (\ref{eqbilinfourmulti}) substantially more singular than Coifman-Meyer multipliers \cite{CoMe}.
An additional structural difficulty observed in \cite{DT} are degeneracies in the way the singular subspace
is positioned with respect to the planes $\xi_1=\xi_2=0$ and $\eta_1=\eta_2=0$.
Some of those degeneracies had to be handled with a mixture of one-dimensional and two-dimensional techniques, while the particular case
\begin{equation}\label{eqmatrices}
A=\begin{bmatrix} 1 & 0 \\ 0 & 0 \end{bmatrix} \quad\textrm{and}\quad B=\begin{bmatrix} 0 & 0 \\ 0 & 1 \end{bmatrix}
\end{equation}
could not be resolved using wave-packet analysis.
Moreover, (\ref{eqmatrices}) was identified as the only such case (modulo trivial transformations) and left out from the discussion.

If one further chooses $m$ associated with a single ``cone'' in the frequency domain, then $T_{\mu}$ will become a paraproduct-like operator
studied by each of the authors in \cite{Ber} and \cite{Kov} respectively.
More precisely, take two Schwartz functions $\varphi$ and $\psi$ such that $\widehat{\psi}(\xi)$ is supported in $\frac{1}{2}\leq|\xi|\leq 2$
and consider the symbol
$$ \mu(\xi_1,\xi_2,\eta_1,\eta_2) = m(\xi_1,\eta_2)
= \sum_{k\in\mathbb{Z}}\widehat{\varphi}(2^{-k}\xi_1)\,\widehat{\psi}(2^{-k}\eta_2) , $$
in which case (\ref{eqbilinfourmulti}) becomes
\begin{equation}\label{eqtwistedpar}
\Pi_{\varphi,\psi}(f,g)(x,y) := \sum_{k\in\mathbb{Z}}
\Big(\int_{\mathbb{R}}f(x-u,y)\,2^k\varphi(2^k u)du\Big) \Big(\int_{\mathbb{R}}g(x,y-v)\,2^k\psi(2^k v)dv\Big) .
\end{equation}
Its continuous-scale variant can be written as
$$ \int_{0}^{\infty} \Big(\int_{\mathbb{R}}f(x-u,y)\,\frac{1}{t}\varphi\Big(\frac{u}{t}\Big)du\Big)
\Big(\int_{\mathbb{R}}g(x,y-v)\,\frac{1}{t}\psi\Big(\frac{v}{t}\Big)dv\Big) \frac{dt}{t} $$
and there is no significant difference between these two models.
Operator (\ref{eqtwistedpar}) is sometimes called the \emph{twisted paraproduct}, due to its paraproduct-like structure.
It satisfies a certain range of $\mathrm{L}^p$ estimates, with constants depending on the first several Schwartz seminorms of $\varphi$ and $\psi$.
We refer to \cite[Theorem 1]{Kov} for the precise formulation of that result.
From the known $\mathrm{L}^p$ bounds for (\ref{eqtwistedpar}) one can then handle $\mathrm{L}^p$ estimates
for the general multipliers $T_\mu$ with $A,B$ as in (\ref{eqmatrices})
and with $m$ satisfying (\ref{eqhomsymbolest}), as was briefly commented in \cite{Kov}.
Some details of that reduction can also be found in Section \ref{sectionproofthm2}.

\smallskip
On the other hand, it is a classical fact that properly supported Coifman-Meyer multipliers satisfy a range of Sobolev norm estimates.
Following the analogy between ordinary paraproducts and the twisted paraproduct we believe it is of interest to formulate and prove
Sobolev estimates for (\ref{eqtwistedpar}) and for the more general class of multipliers to which it belongs.
In order to state our results we denote the two-dimensional non-homogeneous Sobolev norm of $f$ by $\|f\|_{\mathrm{W}^{s,p}}$,
$$ \|f\|_{\mathrm{W}^{s,p}} := \|(I-\triangle)^{s/2}f\|_{\mathrm{L}^{p}(\mathbb{R}^2)} , $$
and we also introduce a particular case of the mixed (or partial) Sobolev norm $\|f\|_{\mathrm{L}^{p}_{y}(\mathrm{W}^{s,p}_{x})}$ defined as
$$ \|f\|_{\mathrm{L}^{p}_{y}(\mathrm{W}^{s,p}_{x})}
:= \big\|\|f(x,y)\|_{\mathrm{W}^{s,p}_{x}(\mathbb{R})}\big\|_{\mathrm{L}^{p}_{y}(\mathbb{R})}
= \|(I-\triangle_x)^{s/2}f\|_{\mathrm{L}^{p}(\mathbb{R}^2)} $$
for $s\geq 0$ and $1<p<\infty$.

Let us agree to write $A\lesssim B$ whenever $A\leq C B$ holds with some constant $C>0$.
The set of parameters on which $C$ depends will usually be clear from the context; otherwise we will state it in the text or denote it in the subscript.
We also write $A\sim B$ when both $A\lesssim B$ and $B\lesssim A$ are satisfied.

As we said, we consider bilinear operators of the form
\begin{equation}\label{eqnewmultiplier}
T_{m}(f,g)(x,y) := \int_{\mathbb{R}^4} m(\xi_1,\eta_2)
e^{2\pi i (x (\xi_1+\eta_1) + y (\xi_2+\eta_2))} \widehat{f}(\xi_1,\xi_2) \widehat{g}(\eta_1,\eta_2) d\xi_1 d\xi_2 d\eta_1 d\eta_2 ,
\end{equation}
defined for Schwartz functions $f,g\in\mathcal{S}(\mathbb{R}^2)$.
It turns out that the simplest nontrivial Sobolev estimate one can hope for is a bound
from $\mathrm{L}^p\times\mathrm{W}^{s,q}$ to $\mathrm{L}^{r}_{y}(\mathrm{W}^{s,r}_{x})$,
at least in the particular case of (\ref{eqtwistedpar}).
Before stating our positive results, let us formulate a simple proposition,
which shows the necessity of imposing some support conditions on general $m$.

\begin{proposition}\label{prop1}
If inequality
$$ \|T_{m}(f,g)\|_{\mathrm{L}^{r}_{y}(\mathrm{W}^{s,r}_{x})} \lesssim \|f\|_{\mathrm{L}^p} \|g\|_{\mathrm{W}^{s,q}} $$
holds for some choice of $1\!<\!p,q,r\!<\!\infty$, $\frac{1}{p}+\frac{1}{q}=\frac{1}{r}$ and $s>0$,
then every super-level set of $m$ has to lie inside a region of the form $|\tau_1|\lesssim 1+|\tau_2|$.
More precisely, for each $\gamma>0$ there exists a constant $c_{\gamma}>0$ such that
$$ \big\{(\tau_1,\tau_2):|m(\tau_1,\tau_2)|\geq\gamma\big\} \,\subseteq\, \big\{(\tau_1,\tau_2):|\tau_1|\leq c_{\gamma}(1+|\tau_2|)\big\} . $$
\end{proposition}

The above proposition also explains the need to break the symmetry between the roles of $f$ and $g$
in order to have nontrivial estimates for $T_{m}(f,g)$.
Since we imposed a homogeneous condition on $m$, we find it natural to discuss only symbols that are actually supported in a half-plane $|\tau_1|\lesssim|\tau_2|$.
Now we can formulate the first boundedness result.

\begin{theorem}\label{theorem2}
Suppose that $m\in\mathrm{C}^{\infty}\big(\mathbb{R}^2\!\setminus\!\{(0,0)\}\big)$ satisfies Estimates \emph{(\ref{eqhomsymbolest})}
and define bilinear multiplier operator $T_{m}$ by \emph{(\ref{eqnewmultiplier})}. If
\begin{equation}\label{eqthm1support}
\mathop{\mathrm{supp}}m \,\subseteq\, \big\{(\tau_1,\tau_2):|\tau_1|\leq c\,|\tau_2|\big\}
\end{equation}
for some constant $c>0$, then the estimate
\begin{equation}\label{eqthm1estimate}
\|T_{m}(f,g)\|_{\mathrm{L}^{r}_{y}(\mathrm{W}^{s,r}_{x})} \lesssim \|f\|_{\mathrm{L}^p} \|g\|_{\mathrm{W}^{s,q}}
\end{equation}
holds for $s\geq 0$ and the exponents $1<p,q,r<\infty$ such that $\frac{1}{p}+\frac{1}{q}=\frac{1}{r}>\frac{1}{2}$.
The implicit constant in \emph{(\ref{eqthm1estimate})} depends on $p,q,r,s$, on constants $C_{\beta_1,\beta_2}$ from \emph{(\ref{eqhomsymbolest})},
and on the constant $c$ from \emph{(\ref{eqthm1support})}.
\end{theorem}

Note that the symbol of the operator (\ref{eqtwistedpar}) fulfills requirement (\ref{eqthm1support}) if $\widehat{\varphi}$ is compactly supported.
Indeed, the proof of Theorem \ref{theorem2} will begin by establishing (\ref{eqthm1estimate}) for a more general variant of (\ref{eqtwistedpar}).
Restricting to the subregion $r<2$ is required in order to be able to apply the results from \cite{Kov},
but we do not know if it is necessary for the estimate to hold.

\smallskip
Finally we turn to $(x,y)$-dependent symbols $\sigma\in\mathrm{C}^{\infty}(\mathbb{R}^4)$ satisfying non-homogeneous estimates of order $0$,
\begin{equation}\label{eqnonhomsymbolest}
\big|\partial_{x}^{\alpha_1}\partial_{y}^{\alpha_2}\partial_{\tau_1}^{\beta_1}\partial_{\tau_2}^{\beta_2}\sigma(x,y,\tau_1,\tau_2)\big|
\leq C_{\alpha_1,\alpha_2,\beta_1,\beta_2} (1+|\tau_1|+|\tau_2|)^{-\beta_1-\beta_2}
\quad\textrm{for all }\alpha_1,\alpha_2,\beta_1,\beta_2\geq 0
\end{equation}
and we define
\begin{equation}\label{eqxydependentoperator}
T_{\sigma}(f,g)(x,y) := \int_{\mathbb{R}^4} \sigma(x,y,\xi_1,\eta_2)
e^{2\pi i (x (\xi_1+\eta_1) + y (\xi_2+\eta_2))} \widehat{f}(\xi_1,\xi_2) \widehat{g}(\eta_1,\eta_2) d\xi_1 d\xi_2 d\eta_1 d\eta_2 .
\end{equation}

\begin{theorem}\label{theorem3}
Suppose that $\sigma\in\mathrm{C}^{\infty}\big(\mathbb{R}^4\big)$ satisfies Estimates \emph{(\ref{eqnonhomsymbolest})}
and define bilinear multiplier operator $T_{\sigma}$ by \emph{(\ref{eqxydependentoperator})}.
If there exists a constant $c>0$ such that for every $(x,y)\in {\mathbb R}^2$
\begin{equation}\label{eqthm1support-bis}
\mathop{\mathrm{supp}} \sigma(x,y,\cdot,\cdot) \subseteq \{(\tau_1,\tau_2)\in\mathbb{R}^2:|\tau_1|\leq c\,|\tau_2|\},
\end{equation}
then the estimate
\begin{equation}\label{eqthm1estimate-bis}
\|T_{\sigma}(f,g)\|_{\mathrm{L}^{r}_{y}(\mathrm{W}^{s,r}_{x})} \lesssim \|f\|_{\mathrm{L}^p} \|g\|_{\mathrm{W}^{s,q}}
\end{equation}
holds for $s\geq 0$ and the exponents $1<p,q,r<\infty$ such that $\frac{1}{p}+\frac{1}{q}=\frac{1}{r}>\frac{1}{2}$.
\end{theorem}

No extra work will be needed for the proof of Theorem \ref{theorem3} as it is derived from the previous result using
a variant the Coifman-Meyer ``freezing argument'' from \cite{CM}.

\smallskip
Let us comment on our motivation for formulating Theorems \ref{theorem2} and \ref{theorem3}.
We wanted to point out analogies between bilinear operators defined by (\ref{eqnewmultiplier}) and (\ref{eqxydependentoperator}) on the one side,
and the Coifman-Meyer multipliers \cite{CoMe},\cite{CM2} and bilinear pseudodifferential operators \cite{BNT},\cite{BT},\cite{B} on the other.
Even though most of the reductions we perform can already be found in the literature, the auxiliary results we rely on
were established by somewhat different techniques in \cite{Ber} and \cite{Kov}, so the remarked connections might not be instantly visible.

\section{Proof of Proposition 1}

We begin with a quite expected $\mathrm{L}^\infty$ bound on the symbol required for any $\mathrm{L}^p$ estimates.

\begin{lemma} \label{lemma} A necessary condition for a bilinear Fourier multiplier $T_\mu$ defined by \emph{(\ref{eqbilinfourmulti})} to be bounded from $\mathrm{L}^p\times\mathrm{L}^q$ to $\mathrm{L}^r$
for some choice of $1<p,q,r<\infty$ such that $\frac{1}{r}=\frac{1}{p}+\frac{1}{q}$ is that its symbol $\mu$ is a.e.\@ bounded. Moreover, we have
\begin{equation}\label{eqsymbollinfty}
\|\mu\|_{\mathrm{L}^\infty} \leq \|T_\mu\|_{\mathrm{L}^p \times \mathrm{L}^q \to \mathrm{L}^r}.
\end{equation}
\end{lemma}

\begin{proof}[Proof of Lemma \ref{lemma}]
Let us fix two points $\xi^0=(\xi^0_1,\xi^0_2)$ and $\eta^0=(\eta^0_1,\eta^0_2)$ in ${\mathbb R}^2$. Choose a Schwartz function $\Phi$ with Fourier transform compactly supported in $[-1,1]^2$ and satisfying $\Phi(0,0)=1$.
Also choose another Schwartz function $\Psi$ such that $\widehat{\Psi}=1$ on $[-2,2]^2$. In particular,
\begin{equation}\label{eqphipsiaux}
\int_{\mathbb{R}^2}\int_{\mathbb{R}^2} \widehat{\Phi}(\xi) \widehat{\Phi}(\eta) \widehat{\Psi}(\xi+\eta) \, d\xi d\eta = \int_{\mathbb{R}^2}\int_{\mathbb{R}^2} \widehat{\Phi}(\xi) \widehat{\Phi}(\eta) \, d\xi d\eta = \Phi(0,0)^2=1.
\end{equation}
Consider the functions defined by
$$ f_\epsilon(x,y) := e^{2\pi i(\xi^0_1 x + \xi^0_2 y)} \Phi(\epsilon x,\epsilon y),
\quad g_\epsilon(x,y) := e^{2\pi i(\eta^0_1 x + \eta^0_2 y)} \Phi(\epsilon x,\epsilon y), $$
and
$$ h_\epsilon(x,y) :=  \epsilon^2 e^{2\pi i\left((\xi^0_1+\eta^0_1)x + (\xi^0_2+\eta^0_2)y\right)} \Psi(\epsilon x,\epsilon y) $$
for $\epsilon>0$. Then we have
$$ \|f_\epsilon\|_{\mathrm{L}^p} \sim \epsilon^{-\frac{2}{p}}, \quad \|g_\epsilon\|_{\mathrm{L}^q} \sim \epsilon^{-\frac{2}{q}}, \quad \textrm{and} \quad \|h_\epsilon\|_{\mathrm{L}^{r'}} \sim \epsilon^{\frac{2}{r}} , $$
so the assumed boundedness of $T_\mu$ implies the estimate
\begin{equation} \left| \langle T_\mu(f_\epsilon , g_\epsilon ),h_\epsilon\rangle \right| \,\lesssim\, \|T_\mu\|_{\mathrm{L}^p \times \mathrm{L}^q \to \mathrm{L}^r} \,\epsilon^{-\frac{2}{p}-\frac{2}{q}+\frac{2}{r}} \,=\, \|T_\mu\|_{\mathrm{L}^p \times \mathrm{L}^q \to \mathrm{L}^r} \label{equni} \end{equation}
uniformly in $\epsilon>0$.
However, using Plancherel's identity we obtain
\begin{align*}
\langle T_\mu(f_\epsilon , g_\epsilon ),h_\epsilon\rangle & = \int_{\mathbb{R}^2}\int_{\mathbb{R}^2} \mu(\xi,\eta) \,\epsilon^{-2}\widehat{\Phi}\big(\epsilon^{-1}(\xi\!-\!\xi^0)\big) \,\epsilon^{-2}\widehat{\Phi}\big(\epsilon^{-1}(\eta\!-\!\eta^0)\big) \,\widehat{\Psi}\big(\epsilon^{-1}(\xi\!+\!\eta\!-\!\xi^0\!-\!\eta^0)\big) \,d\xi d\eta \\
& = \int_{\mathbb{R}^2}\int_{\mathbb{R}^2} \mu(\xi^0+\epsilon\xi ,\eta^0+\epsilon \eta) \,\widehat{\Phi}(\xi) \,\widehat{\Phi}(\eta) \,\widehat{\Psi}(\xi+\eta) \, d\xi d\eta,
\end{align*}
which in a combination with (\ref{eqphipsiaux}) proves that for almost every $(\xi^0,\eta^0)\in\mathbb{R}^2\times\mathbb{R}^2$
$$ \lim_{\epsilon\to 0} \  \langle T_\mu(f_\epsilon , g_\epsilon ),h_\epsilon\rangle = \mu(\xi^0,\eta^0). $$
From (\ref{equni}) we finally obtain (\ref{eqsymbollinfty}).
\end{proof}

\begin{proof}[Proof of Proposition \ref{prop1}]
Observe that boundedness of $T_m$ from  ${\mathrm{L}^p} \times {\mathrm{W}^{s,q}}$ to $\mathrm{L}^{r}_{y}(\mathrm{W}^{s,r}_{x})$ is equivalent to boundedness of $T_{\nu}$ from ${\mathrm{L}^p} \times {\mathrm{L}^q}$ to $\mathrm{L}^{r}$,
where $T_\nu$ is a multiplier with the new symbol
$$ \nu(\xi_1,\xi_2,\eta_1,\eta_2):=m(\xi_1,\eta_2) \left(\frac{1+4\pi^2(\xi_1+\eta_1)^2}{1+4\pi^2(\eta_1^2+\eta_2^2)}\right)^{s/2}. $$
Applying Lemma \ref{lemma} we get that
$$ \|\nu\|_{\mathrm{L}^\infty} \leq \|T_\nu\|_{{\mathrm{L}^p} \times {\mathrm{L}^{q}} \to \mathrm{L}^{r}}
=\|T_m\|_{{\mathrm{L}^p} \times {\mathrm{W}^{s,q}} \to \mathrm{L}^{r}_{y}(\mathrm{W}^{s,r}_{x})} =: C_m .$$
In particular, for $\xi_1,\eta_2$ satisfying $|m(\xi_1,\eta_2)|\geq \gamma>0$ by choosing $\eta_1=0$ we obtain
$$ \frac{|\xi_1|}{1+|\eta_2|} \lesssim \left(\frac{1+4\pi^2\xi_1^2}{1+4\pi^2\eta_2^2}\right)^{1/2} \leq \left(C_m \gamma^{-1}\right)^{1/s},$$
which yields the desired support condition.
\end{proof}

\section{Proof of Theorem 2}
\label{sectionproofthm2}

As we have already announced, the first step is to establish (\ref{eqthm1estimate}) for a slight variation of the twisted paraproduct. Let
$$ \|\varphi\|_{\alpha,\beta} := \sup_{t\in\mathbb{R}}\,|t|^\alpha |\partial^{\beta}_{t}\varphi(t)|
\quad\textrm{for integers } \alpha,\beta\geq 0 $$
denote the Schwartz seminorms of $\varphi\in\mathcal{S}(\mathbb{R})$.
For any $\varphi,\psi\in\mathcal{S}(\mathbb{R})$ and for a family of complex coefficients $\lambda=(\lambda_k)_{k\in\mathbb{Z}}$
we define a variant of (\ref{eqtwistedpar}) by
\begin{equation}\label{eqtwistedpar2}
\Pi_{\varphi,\psi,\lambda}(f,g) := \sum_{k\in\mathbb{Z}} \lambda_k \,\mathrm{P}^{\varphi_k}_{x}f \ \mathrm{P}^{\psi_k}_{y}g ,
\end{equation}
where
$$ \varphi_k(t)=2^k\varphi(2^k t), \quad \psi_k(t)=2^k\psi(2^k t) $$
and
$$ \mathrm{P}^{\varphi_k}_{x}f(x,y) = \int_{\mathbb{R}}f(x-t,y)\varphi_k(t)dt,
\quad \mathrm{P}^{\psi_k}_{y}g(x,y) = \int_{\mathbb{R}}g(x,y-t)\psi_k(t)dt . $$

For the purpose of performing Littlewood-Paley decompositions, we fix a nonnegative, even, smooth function $\theta$
which is also decreasing on $[0,\infty)$, equals $1$ on $[0,\frac{1}{2}]$, and equals $0$ on $[1,\infty)$.
It will be understood that all constants depend on the choice of $\theta$.
Define $\vartheta$ by $\vartheta(\tau)=\theta(\frac{\tau}{2})-\theta(\tau)$.
Note that $\theta(2^{-k}\tau)$ is supported in $|\tau|\leq 2^k$ and that
$\vartheta(2^{-k}\tau)$ is supported in $2^{k-1}\!\!\leq\!|\tau|\!\leq\!2^{k+1}$.
Write $\Delta_k$ for a multiplier associated with the symbol $\vartheta(2^{-k}\tau)$.
Each Littlewood-Paley projection we use will be one-dimensional and we will denote the variable to which it relates in a subscript.

\begin{lemma}\label{lemma1}
Suppose that $\widehat{\varphi}(\xi)$ is supported in $|\xi|\leq c$ for some $c>0$ and that $\widehat{\psi}(\xi)$ is supported in $\frac{1}{2}\leq|\xi|\leq 2$.
Operators \emph{(\ref{eqtwistedpar2})} satisfy
$$ \|\Pi_{\varphi,\psi,\lambda}(f,g)\|_{\mathrm{L}^{r}_{y}(\mathrm{W}^{s,r}_{x})}
\,\lesssim_{c,p,q,r,s} \Big(\max_{0\leq\alpha,\beta\leq 3}\|\varphi\|_{\alpha,\beta}\Big) \Big(\max_{0\leq\alpha,\beta\leq 3}\|\psi\|_{\alpha,\beta}\Big)
\Big(\sup_{k\in\mathbb{Z}}|\lambda_k|\Big) \,\|f\|_{\mathrm{L}^p} \|g\|_{\mathrm{W}^{s,q}} $$
whenever \,$1<p,q,r<\infty$, \,$\frac{1}{p}+\frac{1}{q}=\frac{1}{r}>\frac{1}{2}$, \,$s\geq 0$.
\end{lemma}

We will present two slightly different proofs.
The first one will be a bit more elegant but it will prefer integer values of $s$, while the second one will address general $s\geq 0$ directly and avoid the use of interpolation.

\begin{proof}[First proof of Lemma \ref{lemma1}]
By linearity of Expression (\ref{eqtwistedpar2}) in $\varphi$, $\psi$, and $\lambda$ we may assume that\linebreak
$\|\varphi\|_{\alpha,\beta},\|\psi\|_{\alpha,\beta}\leq 1$ for $0\leq\alpha,\beta\leq 3$
and that $|\lambda_k|\leq 1$ for $k\in\mathbb{Z}$.
Next, we make use of non-isotropic dilations
$$ \mathrm{D}_{a}h(x,y) := h(2^{-a}x,y) $$
to reduce to the case when $c=1$. Indeed, if $a$ is an integer such that $2^a\geq c$, then by writing
$$ \mathrm{D}_{a}\Pi_{\varphi,\psi,\lambda}(f,g)
= \sum_{k\in\mathbb{Z}} \lambda_k \ \mathrm{P}^{\varphi_{k-a}}_{x}\mathrm{D}_{a}f \ \mathrm{P}^{\psi_k}_{y}\mathrm{D}_{a}g $$
we effectively replace $\varphi(t)$ with $\varphi_{-a}(t)=2^{-a}\varphi(2^{-a}t)$ in the desired estimate.
The Sobolev norms in question can change up to a constant that depends on $s$ and $c$, but this is allowed.
It remains to observe that $\widehat{\varphi}_{-a}(\xi)$ is now supported in $|\xi|\leq 1$.

Because of $s\geq 0$ we have
$$ \|h\|_{\mathrm{L}^{r}_{y}(\mathrm{W}^{s,r}_{x})} \sim_{r,s}
\big\|\|h\|_{\mathrm{L}^{r}_{x}(\mathbb{R})}+\|(-\triangle_x)^{s/2}h\|_{\mathrm{L}^{r}_{x}(\mathbb{R})}\big\|_{\mathrm{L}^{r}_{y}(\mathbb{R})}
\sim_{r,s} \|h\|_{\mathrm{L}^{r}(\mathbb{R}^2)} + \|\partial^{s}_{x}h\|_{\mathrm{L}^{r}(\mathbb{R}^2)} , $$
so one actually needs to bound $\mathrm{L}^r$ norms of
$\Pi_{\varphi,\psi,\lambda}(f,g)$ and $\partial^{s}_{x}\Pi_{\varphi,\psi,\lambda}(f,g)$.
In this proof we find convenient to work with integer values of $s$.
Arguments very similar to the complex interpolation of Sobolev spaces can then extend the inequality to general nonnegative $s$;
for example see \cite{BL}.

The estimate
\begin{equation}\label{eqproof1order0}
\|\Pi_{\varphi,\psi,\lambda}(f,g)\|_{\mathrm{L}^{r}} \lesssim_{p,q,r}
\|f\|_{\mathrm{L}^p} \|g\|_{\mathrm{L}^q} \lesssim_s \|f\|_{\mathrm{L}^p} \|g\|_{\mathrm{W}^{s,q}}
\end{equation}
actually follows from \cite[Theorem 1]{Kov}.
The only difference is the existence of coefficients $\lambda_k$ in the definition of (\ref{eqtwistedpar2}).
However, one can insert arbitrary bounded coefficients in Equation (6.6) of \cite[Section 6]{Kov}
and hence arbitrary bounded coefficients can also appear in the definition of $T_\mathrm{c}$ in \cite[Section 1]{Kov}.

We turn to bounding the higher derivatives and for this we apply the product rule,
\begin{equation}\label{eqproof1productexp}
\partial^{s}_{x}\Pi_{\varphi,\psi,\lambda}(f,g)(x,y)
= \sum_{\beta =0}^{s} {s \choose \beta } \sum_{k\in\mathbb{Z}} \lambda_k \,\partial^{\beta }_{x}\mathrm{P}^{\varphi_k}_{x}f(x,y)
\,\mathrm{P}^{\psi_k}_{y}\partial^{s-\beta }_{x}g(x,y) .
\end{equation}
The term for $\beta =0$ is precisely \,$\Pi_{\varphi,\psi,\lambda}(f,\partial^{s}_{x}g)$\, and (\ref{eqproof1order0}) implies
\begin{equation}\label{eqproof1orderspart1}
\|\Pi_{\varphi,\psi,\lambda}(f,\partial^{s}_{x}g)\|_{\mathrm{L}^{r}} \lesssim_{p,q,r}
\|f\|_{\mathrm{L}^p} \|\partial^{s}_{x}g\|_{\mathrm{L}^q} \lesssim_s \|f\|_{\mathrm{L}^p} \|g\|_{\mathrm{W}^{s,q}} .
\end{equation}
In the rest we work with a fixed integer $\beta $ such that $1\leq \beta \leq s$.
We have
$$ \partial^{\beta }_{x}\mathrm{P}^{\varphi_k}_{x}f
= \sum_{j=-\infty}^{0} \partial^{\beta }_{x}\Delta_{k+j,x}\mathrm{P}^{\varphi_k}_{x}f
= \sum_{j=-\infty}^{0} 2^{\beta (k+j)} \mathrm{P}^{\tilde{\psi}_{k,j}}_{x}f , $$
where the functions $\tilde{\psi}_{k,j}$ are defined by
$$ (\tilde{\psi}_{k,j})\hat{\rule{0mm}{2mm}}(\xi) = (2\pi i)^{\beta }\Big(\frac{\xi}{2^{k+j}}\Big)^{\beta }\vartheta(2^{-k-j}\xi)\widehat{\varphi}_{k}(\xi) . $$
Observe that the Fourier transform of $\tilde{\psi}_{k,j}$ is a smooth function adapted to the annulus $2^{k+j-1}\leq|\xi|\leq 2^{k+j+1}$.
After expanding
$$ \sum_{k\in\mathbb{Z}} \lambda_k\, \partial^{\beta }_{x}\mathrm{P}^{\varphi_k}_{x}f
\ \mathrm{P}^{\psi_k}_{y}\partial^{s-\beta }_{x}g
= \sum_{j=-\infty}^{0}2^{\beta j}\,\sum_{k\in\mathbb{Z}} \lambda_k\, \mathrm{P}^{\tilde{\psi}_{k,j}}_{x}f
\ 2^{\beta k}\mathrm{P}^{\psi_k}_{y}\partial^{s-\beta }_{x}g $$
we apply the Cauchy-Schwarz inequality in $k\in\mathbb{Z}$.
Using the classical Littlewood-Paley inequalities in $x$ and $y$ variables respectively, we obtain
$$ \Big\|\Big(\sum_{k\in\mathbb{Z}}\big|\mathrm{P}^{\tilde{\psi}_{k,j}}_{x}f\big|^{2}\Big)^{1/2}\Big\|_{\mathrm{L}^p}
\lesssim \|f\|_{\mathrm{L}^p} $$
and
$$ \Big\|\Big(\sum_{k\in\mathbb{Z}}\big|2^{\beta k}\mathrm{P}^{\psi_k}_{y}\partial^{s-\beta }_{x}g\big|^2\Big)^{1/2}\Big\|_{\mathrm{L}^{q}}
\lesssim \big\|\partial^{\beta }_{y}\partial^{s-\beta }_{x}g\big\|_{\mathrm{L}^q} , $$
so summing in $j$ finally gives
\begin{equation}\label{eqproof1orderspart2}
\Big\|\sum_{k\in\mathbb{Z}} \lambda_k \,\partial^{\beta }_{x}\mathrm{P}^{\varphi_k}_{x}f
\,\mathrm{P}^{\psi_k}_{y}\partial^{s-\beta }_{x}g\Big\|_{\mathrm{L}^{r}} \lesssim_{p,q,r,s,\beta}
\|f\|_{\mathrm{L}^p} \|\partial^{\beta }_{y}\partial^{s-\beta }_{x}g\|_{\mathrm{L}^q} .
\end{equation}
Combining Expansion (\ref{eqproof1productexp}) with Estimates (\ref{eqproof1order0}), (\ref{eqproof1orderspart1}), and (\ref{eqproof1orderspart2})
we complete the proof of Lemma \ref{lemma1}.
\end{proof}

\begin{proof}[Alternative proof of Lemma \ref{lemma1}]
Begin with the same normalizations and reductions as before.
This time we let $\Delta_k$ denote a general smooth $1$-dimensional Littlewood-Paley truncation at frequency scale $2^k$.
We do not insist on the fixed choice of truncations from the beginning of this section, but rather allow the smooth cutoffs to change
from expression to expression and from line to line, and still use the generic notation $\Delta_k$.
Whenever we have an equality that holds up to the change of the Littlewood-Paley cutoffs, we write it using the sign $\simeq$.

Because of $s\geq 0$, we have
$$ \|h\|_{\mathrm{L}^{r}_{y}(\mathrm{W}^{s,r}_{x})} \sim_{r,s} \|h\|_{\mathrm{L}^{r}(\mathbb{R}^2)} +
\Big\| \Big(\sum_{\ell \geq 0} 2^{2\ell s} |\Delta_{\ell,x} h|^2\Big)^{\frac{1}{2}} \Big\|_{\mathrm{L}^{r}(\mathbb{R}^2)}. $$
Since the $\mathrm{L}^r$ norm of $\Pi_{\varphi,\psi,\lambda}(f,g)$ is controlled by (\ref{eqproof1order0}), one actually needs to study\linebreak
$\Delta_{\ell,x}\Pi_{\varphi,\psi,\lambda}(f,g)$.
As usually, for two functions $h_1,h_2$ we have
\begin{align*}
 \Delta_{\ell}(h_1 h_2) & = \sum_{j_1,j_2} \Delta_{\ell } \left(\Delta_{j_1} (h_1) \Delta_{j_2} (h_2)\right) \\
  & \simeq \sum_{|\max(j_1,j_2) -\ell|\leq 4} \Delta_{\ell } \left(\Delta_{j_1} (h_1) \Delta_{j_2} (h_2)\right) + \sum_{\substack{|j_1-j_2|\leq 4\\ \ell \leq j_1,j_2}} \Delta_{\ell } \left(\Delta_{j_1} (h_1) \Delta_{j_2} (h_2)\right).
  \end{align*}
Applying the above decomposition to $\Delta_{\ell,x}\Pi_{\varphi,\psi,\lambda}(f,g)$ we obtain these three terms:
$$ I_\ell^1:= \sum_{k} \sum_{j_2\leq j_1 \sim \ell}\lambda_k \,\Delta_{\ell,x}\left(\Delta_{j_1,x}\mathrm{P}^{\varphi_k}_{x}f \ \Delta_{j_2,x}\mathrm{P}^{\psi_k}_{y}g\right),$$
$$ I_\ell^2:= \sum_{k} \sum_{j_1\leq j_2 \sim \ell}\lambda_k \,\Delta_{\ell,x}\left(\Delta_{j_1,x}\mathrm{P}^{\varphi_k}_{x}f \ \Delta_{j_2,x}\mathrm{P}^{\psi_k}_{y}g\right),$$
and
$$ I_\ell^3:= \sum_{k} \sum_{\substack{|j_1-j_2|\leq 4\\ \ell \leq j_1,j_2}}\lambda_k \,\Delta_{\ell,x}\left(\Delta_{j_1,x}\mathrm{P}^{\varphi_k}_{x}f \ \Delta_{j_2,x}\mathrm{P}^{\psi_k}_{y}g\right).$$

\smallskip
\noindent
\emph{Term $I_\ell^1$.}
We observe that $\Delta_{j_1,x}\mathrm{P}^{\varphi_k}_{x} \simeq \Delta_{j_1,x}$ for $j_1\leq k$ and that it vanishes if $j_1> k$. Thus, the sum in $I_1$ may be reduced to the following model sum,
$$ I_\ell^1\simeq \sum_{k\geq \ell}  \sum_{j_2\leq \ell}\lambda_k \,\Delta_{\ell,x}\left(\Delta_{\ell,x}f \ \Delta_{j_2,x}\mathrm{P}^{\psi_k}_{y}g\right),$$
which implies
$$ 2^{\ell s} |I_\ell^1|\lesssim \sum_{k\geq \ell} 2^{(\ell-k)s}\big|\Delta_{\ell,x}\left(\Delta_{\ell,x} f \ 2^{k s}\mathrm{P}^{\phi_\ell}_{x}\mathrm{P}^{\psi_k}_{y}g\right)\big|.$$
Here we use the generic notation for $\phi$ and the associated convolution operator. By usual estimates involving the Hardy-Littlewood maximal function $\M$ we get
$$ 2^{\ell s} |I_\ell^1|\lesssim \sum_{k\geq \ell} 2^{(\ell-k)s} \M_{x}\left( \M_{x} f  \ 2^{k s}\M_x \mathrm{P}^{\psi_k}_{y} g\right).$$
Hence for $s>0$ by Young's inequality,
$$ \left\| 2^{\ell s} I_\ell^1 \right\|_{\ell^2(\ell\geq 0)} \lesssim \left\| \M_{x}\left( \M_{x} f \ 2^{k s}\M_x \mathrm{P}^{\psi_k}_{y}g\right) \right\|_{\ell^2(k)}.$$
Using the $\ell^2$-valued inequality for the maximal function we obtain
\begin{align*}
 \left\| \| 2^{\ell s} I_\ell^1 \|_{\ell^{2}(\ell\geq 0)} \right\|_{\mathrm{L}^r} & \lesssim  \| f\|_{\mathrm{L}^p}  \left\| \| 2^{k s}\mathrm{P}^{\psi_k}_{y}g\|_{\ell^2(k\geq 0)}\right\|_{\mathrm{L}^q} \\
 & \lesssim \|f\|_{\mathrm{L}^p} \|g\|_{\mathrm{L}^q(W^{s,q}_y)} \lesssim \|f\|_{\mathrm{L}^p} \|g\|_{W^{s,q}}.
 \end{align*}

\smallskip
\noindent
\emph{Term $I_\ell^2$.}
By similar considerations we observe that this term can be reduced to the following model sum,
$$ I_\ell^2\simeq  \sum_{k\geq j_1} \sum_{j_1\leq \ell}\lambda_k \,\Delta_{\ell,x}\left(\Delta_{j_1,x}\mathrm{P}^{\varphi_k}_{x}f \ \Delta_{\ell,x}\mathrm{P}^{\psi_k}_{y}g\right).$$
Computing the sum over $j_1$ and $\ell$ it may be reduced to a twisted paraproduct applied to $f$ and $\Delta_{\ell,x}g$,
$$ I_\ell^2 \simeq \Pi_{\varphi,\psi,\lambda}(f, \Delta_{\ell,x} g) .$$
Using the $\mathrm{L}^p \times \mathrm{L}^q \to \mathrm{L}^r$ boundedness of $\Pi_{\varphi,\psi,\lambda}$ we also know that it admits $\ell^2$-valued estimates. Hence,
\begin{align*}
 \left\| \| 2^{\ell s} I_\ell^2 \|_{\ell^{2}(\ell\geq 0)} \right\|_{\mathrm{L}^r} & \lesssim   \left\| \| 2^{\ell s}\Pi_{\varphi,\psi,\lambda}(f, \Delta_{\ell,x} g) \|_{\ell^{2}(\ell\geq 0)} \right\|_{\mathrm{L}^r} \\
 &   \lesssim \| f\|_{\mathrm{L}^p}  \left\| \| 2^{\ell s} \Delta_{\ell,x} g\|_{\ell^2(\ell\geq 0)}\right\|_{\mathrm{L}^q} \\
 & \lesssim \|f\|_{\mathrm{L}^p} \|g\|_{\mathrm{L}^q(W^{s,q}_x)} \lesssim \|f\|_{\mathrm{L}^p} \|g\|_{W^{s,q}}.
 \end{align*}

\smallskip
\noindent
\emph{Term $I_\ell^3$.}
By similar considerations we observe that this term can be reduced to the following model sum,
$$ I_\ell^3 \simeq \sum_{\ell \leq j\leq k} \lambda_k \,\Delta_{\ell,x}\left(\Delta_{j,x}\mathrm{P}^{\varphi_k}_{x}f \ \Delta_{j,x}\mathrm{P}^{\psi_k}_{y}g\right).$$
Computing the sum over $k$ we obtain
\begin{align*}
 I_\ell^3 & \simeq \sum_{\ell \leq j\leq k} \lambda_k \, \Delta_{\ell,x}\Big(\Delta_{j,x}f \ \Delta_{j,x}\mathrm{P}^{\psi_k}_{y}g\Big) \\
  & \simeq \sum_{\ell \leq j }  \Delta_{\ell,x}\Big(\Delta_{j,x}f \ \Delta_{j,x} \Big(\sum_{j\leq k} \lambda_k \mathrm{P}^{\psi_k}_{y}\Big)g\Big).
\end{align*}
Reasoning as for the first term gives
\begin{align*}
 \left\| \| 2^{s\ell}I_\ell^3\|_{\ell^2(\ell)}\right\|_{\mathrm{L}^r}   & \lesssim  \|f\|_{\mathrm{L}^p} \bigg\| \Big\| 2^{js }\Big(\sum_{j\leq k} \lambda_k \mathrm{P}^{\psi_k}_{y}\Big)g\Big\|_{\ell^2(j)} \bigg\|_{\mathrm{L}^q} \\
   & \lesssim  \|f\|_{\mathrm{L}^p} \left\| \| 2^{ks } \mathrm{P}^{\psi_k}_{y}g\|_{\ell^2(k)} \right\|_{\mathrm{L}^q} \\
   & \lesssim \|f\|_{\mathrm{L}^p} \|g\|_{W^{s,q}}.
   \end{align*}

\smallskip
\noindent
This completes the proof of Sobolev estimates for the three terms and thus also establishes Lemma \ref{lemma1}.
\end{proof}

Now we are ready to prove the general result by a somewhat standard decomposition of (\ref{eqnewmultiplier})
into a rapidly convergent series of twisted paraproducts (\ref{eqtwistedpar2}).

\begin{proof}[Proof of Theorem \ref{theorem2}]
Let $a$ be the smallest nonnegative integer such that $c\leq 2^a$.
For $(\tau_1,\tau_2)\in\mathop{\mathrm{supp}}m$ we have by (\ref{eqthm1support}),
$$ 1 = \sum_{\substack{j,k\in\mathbb{Z}\\ j\leq k+a+2}} \vartheta(2^{-j}\tau_1) \,\vartheta(2^{-k}\tau_2)
= \sum_{k\in\mathbb{Z}} \theta(2^{-k-a-3}\tau_1) \,\vartheta(2^{-k}\tau_2) , $$
because the term in the double sum can be nonzero only when
$$ 2^{j-1} \leq |\tau_1| \leq 2^a |\tau_2| \leq 2^{k+a+1} . $$
Multiplying by $m$ we obtain $m=\sum_{k\in\mathbb{Z}}m_k$, where
$$ m_k(\tau_1,\tau_2) = m(\tau_1,\tau_2) \,\theta(2^{-k-a-3}\tau_1) \,\vartheta(2^{-k}\tau_2) . $$
Obviously, $m_k$ is supported in
$$ \big\{(\tau_1,\tau_2)\in\mathbb{R}^2 \,:\, |\tau_1|\leq 2^{k+a+3}, \ 2^{k-1}\!\leq|\tau_2|\leq2^{k+1} \big\} , $$
so let us expand it into a Fourier series on $[-2^{k+a+4},2^{k+a+4}]^2$,
$$ m_k(\tau_1,\tau_2) = \sum_{n_1,n_2\in\mathbb{Z}} \kappa_{n_1,n_2}^{(k)} e^{i\pi 2^{-k-a-4}(n_1\tau_1+n_2\tau_2)} . $$
By (\ref{eqhomsymbolest}) and the product rule we have
$$ \big|\partial_{\tau_1}^{\beta_1}\partial_{\tau_2}^{\beta_2}m_k(\tau_1,\tau_2)\big|
\lesssim_{\beta_1,\beta_2} 2^{-(\beta_1+\beta_2)k} . $$
Since the Fourier coefficients satisfy
$$ n_1^{\beta_1}n_2^{\beta_2}\kappa_{n_1,n_2}^{(k)} = \Big(\frac{i\,2^{k+a+4}}{\pi}\Big)^{\beta_1+\beta_2}
-\!\!\!\!\!\!\int_{[-2^{k+a+4},2^{k+a+4}]^2} \partial_{\tau_1}^{\beta_1}\partial_{\tau_2}^{\beta_2}m_k(\tau_1,\tau_2)
e^{i\pi 2^{-k-a-4}(n_1\tau_1+n_2\tau_2)} d\tau_1 d\tau_2 , $$
we have obtained
\begin{equation}\label{eqcoeffest}
|\kappa_{n_1,n_2}^{(k)}| \,\lesssim (1+|n_1|+|n_2|)^{-10} .
\end{equation}

Define $\varphi^{(n_1)}$ and $\psi^{(n_2,i)}$ by
\begin{align*}
& \widehat{\varphi}^{(n_1)}(\tau) = \theta(2^{-a-4}\tau)e^{i\pi 2^{-a-4}n_1\tau}, \\
& \widehat{\psi}^{(n_2,i)}(\tau) = \vartheta(2^{-i}\tau) e^{i\pi 2^{-a-4}n_2\tau} \quad \textrm{for } i=-1,0,1 ,
\end{align*}
so that
$$ m(\tau_1,\tau_2) = \sum_{i\in\{-1,0,1\}}\sum_{n_1,n_2\in\mathbb{Z}} \sum_{k\in\mathbb{Z}} \,\kappa^{(k)}_{n_1,n_2}
\,\widehat{\varphi}^{(n_1)}(2^{-k}\tau_1) \,\widehat{\psi}^{(n_2,i)}(2^{-k}\tau_2) . $$
Finally, observe that the norms $\|\varphi^{(n_1)}\|_{\alpha,\beta}$, $\|\psi^{(n_2,i)}\|_{\alpha,\beta}$, $0\leq\alpha,\beta\leq 3$
grow at most with the third powers of $n_1$ and $n_2$, while the coefficients $\kappa_{n_1,n_2}^{(k)}$ decay much faster.
This realizes $T_m$ as a convergent sum of paraproducts (\ref{eqtwistedpar2}) and it remains to apply Lemma \ref{lemma1}.
\end{proof}

\section{Proof of Theorem 3}

In this section we only list the main ingredients and give references, as the proof closely follows excerpts from the existing literature.

\begin{proof}[Proof of Theorem \ref{theorem3}]
The way to get boundedness for symbols with smooth dependence on spatial variables is now well-known and it uses the Coifman-Meyer ``freezing argument'' \cite{CM}. Therefore we only sketch the proof and refer the reader to \cite[Theorem 2.4]{B} for a very similar argument adapted to some other bilinear Sobolev estimates.
The procedure relies on local estimates and we only have to check the following local version of (\ref{eqthm1estimate}): For every two-dimensional ball $B$ of radius $1$ and every $(x,y)$-independent symbol
$$ \sigma(x,y,\tau_1,\tau_2) = m(\tau_1,\tau_2) $$
satisfying (\ref{eqnonhomsymbolest}) one has
\begin{equation}\label{eqthm1estimateloc}
\|T_{m}(f,g)\|_{\mathrm{L}^{r}_{y}(\mathrm{W}^{s,r}_{x})(B)} \leq \sum_{j_1,j_2\geq 0} \gamma_{j_1,j_2} \|f\|_{\mathrm{L}^p(2^{j_1} B)} \|g\|_{\mathrm{W}^{s,q}(2^{j_2}B)},
\end{equation}
with some rapidly decreasing sequence of coefficients $\gamma_{j_1,j_2}$.

Following the reduction in the proof of Theorem \ref{theorem2} it is sufficient to check (\ref{eqthm1estimateloc}) for modified twisted paraproducts of the form $\Pi_{\varphi,\psi,\lambda}$, which correspond to the symbols
$$ m(\xi_1,\eta_2) =  \sum_{k\geq 0} \lambda_k \, \widehat{\varphi}(2^{-k}\xi_1)\widehat{\psi}(2^{-k}\eta_2).$$
Here we restrict the sum to nonnegative indices $k$, corresponding to frequencies at scales larger than $1$, since we are only concerned with non-homogeneous regularity condition on the symbol.
However, the kernel of this operator (given by the inverse Fourier transform of the symbol) is of the form $K_m\big((y_1,y_2),(z_1,z_2)\big)=K(y_1,z_2)$ and it satisfies
$$ \left| K(y_1,z_2)  \right| \lesssim (|y_1|+|z_2|)^{-N}$$
for an arbitrarily large exponent $N$. Thus it is easy to check that the bilinear operator satisfies (\ref{eqthm1estimateloc}) with coefficients
$$ |\gamma(j_1,j_2)| \sim (2^{j_1}+2^{j_2})^{-N}.$$
Indeed, the diagonal part \,$j_1,j_2\leq 4$\, is resolved by Theorem \ref{theorem2} and then the non-diagonal part \,$\max\{j_1,j_2\}>4$\, is an easy consequence of the pointwise decay of the kernel.

As a consequence, it is not difficult to see that (\ref{eqthm1estimateloc}) holds for every $(x,y)$-independent symbol satisfying (\ref{eqnonhomsymbolest}). Therefore this estimate also holds for $(x,y)$-dependent symbols, due to the Coifman-Meyer freezing argument, as we have already commented. Finally, summing these local estimates over a covering of the whole space by balls of radius $1$, we obtain the required global estimate.
\end{proof}

\end{document}